\documentclass[a4paper,11pt,reqno]{amsart}

\usepackage{url}
\usepackage[colorlinks=true, linkcolor=blue, citecolor=ForestGreen]{hyperref}
\usepackage{cite}

\usepackage{latexsym,amsmath,amssymb,amsfonts,amsthm}
\usepackage{mathrsfs}
\usepackage{mathtools}
\usepackage{dsfont}
\usepackage{bbm}
\usepackage{soul}
\usepackage{cancel}
\usepackage[utf8]{inputenc}

\usepackage[usenames, dvipsnames]{color}
\usepackage[usenames, dvipsnames, cmyk]{xcolor}
\usepackage{graphicx}
\usepackage[scriptsize,hang,raggedright]{subfigure}
\usepackage{multirow}
\usepackage{enumerate}
\usepackage{enumitem}

\usepackage{a4wide}
\usepackage{epsfig,epstopdf}

\numberwithin{equation}{section}

\definecolor{grey}{rgb}{.7,.7,.7}
\definecolor{refkey}{gray}{.45}
\definecolor{labelkey}{gray}{.45}

\newtheorem{theorem}{Theorem}[section]
\newtheorem{proposition}[theorem]{Proposition}
\newtheorem{lemma}[theorem]{Lemma}

\newtheorem*{mainresult}{Main Result}

\theoremstyle{remark}
\newtheorem{remark}[theorem]{Remark}
\theoremstyle{definition}
\newtheorem{definition}[theorem]{Definition}
\newtheorem{example}[theorem]{Example}

\usepackage{tikz}
\usepackage{pgf,pgfplots}
\usetikzlibrary{calc}
\usetikzlibrary{decorations.markings}
\usetikzlibrary{decorations.pathmorphing}
\usetikzlibrary{decorations.shapes}
\usetikzlibrary{shapes,arrows,shapes.geometric,patterns,fadings}

\newcommand{\R}{\mathbb R}

\newcommand{\dd}{\,\mathrm{d}}
\renewcommand{\setminus}{\backslash}

\newcommand{\defeq}{\coloneqq}

\newcommand{\pol}{\mathscr{P}}
\newcommand{\p}{\mathcal{P}}
\newcommand{\side}[2]{\mkern2mu \overline{\mkern-2mu #1 #2}}

\newcommand{\ba}{\begin{array}}
\newcommand{\ea}{\end{array}}

\newcommand{\bthm}{\begin{theorem}}
\newcommand{\ethm}{\end{theorem}}
\newcommand{\bprop}{\begin{proposition}}
\newcommand{\eprop}{\end{proposition}}
\newcommand{\blemma}{\begin{lemma}}
\newcommand{\elemma}{\end{lemma}}
\newcommand{\bexmpl}{\begin{example}}
\newcommand{\eexmpl}{\end{example}}

\newcommand{\beqn}{\begin{equation}}
\newcommand{\eeqn}{\end{equation}}
\newcommand{\beqns}{\begin{equation*}}
\newcommand{\eeqns}{\end{equation*}}

\newcommand{\pr}{\prime}
\newcommand{\pt}{\partial}

\renewcommand{\geq}{\geqslant}

\definecolor{mygreen}{rgb}{0.1,0.75,0.2}

\newcounter{myenumi}
\DeclareMathOperator{\Per}{Per}

\DeclareMathOperator{\per}{Per}
\DeclareMathOperator{\capp}{cap}

\allowdisplaybreaks

\title[Alexandrov Theorem for Polygons]{Alexandrov's Soap Bubble Theorem for Polygons}

\author{Marco Bonacini}
\address[Marco Bonacini]{Department of Mathematics, University of Trento, Italy}
\email{marco.bonacini@unitn.it}

\author{Riccardo Cristoferi}
\address[Riccardo Cristoferi]{Department of Mathematics - IMAPP, Radboud University, Nijmegen, The Netherlands}
\email{riccardo.cristoferi@ru.nl}

\author{Ihsan Topaloglu}
\address[Ihsan Topaloglu]{Department of Mathematics and Applied Mathematics, Virginia Commonwealth University, Richmond, VA, USA}
\email{iatopaloglu@vcu.edu}

\date{\today}                                        

\thanks{This is an original manuscript of an article that has been accepted for publication in the American Mathematical Monthly, published by Taylor \& Francis.}

\begin{document}

\begin{abstract}
Regular polygons are characterized as area-constrained critical points of the perimeter functional with respect to particular families of perturbations in the class of polygons with a fixed number of sides. We also review recent results in the literature involving other shape functionals as well as further open problems.
\end{abstract}

\maketitle

\section{Introduction}\label{sec:intro}

Aleksandr Danilovich Alexandrov's Soap Bubble Theorem, as proved in \cite{Ale62_2,Ale62_1}, states that a compact, connected embedded hypersurface with constant mean curvature in the Euclidean space $\R^d$ must be a sphere.

This characterization of the sphere is closely linked to the \emph{isoperimetric property} of the Euclidean ball: among all measurable sets in $\R^d$ having the same volume ($d$-dimensional Lebesgue measure), the Euclidean ball uniquely minimizes the perimeter functional (understood here in the sense of Renato Caccioppoli and Ennio De Giorgi, see \cite{Fus04} for an extensive review). The bridge connecting Alexandrov's Theorem and the Isoperimetric Problem is established through a cornerstone principle in the Calculus of Variations: a minimizing set $E$ must satisfy the first order necessary condition (\emph{criticality}, or \emph{stationarity}). This condition is obtained by considering one-parameter families of competitors $\{E_t\}_{t\in\R}$, with $E_0=E$, and imposing that the \emph{first variation} of the functional vanishes along any such volume-preserving perturbation. For the perimeter functional, the condition entails
\begin{equation} \label{intro-1}
    \frac{\dd}{\dd t}\Big|_{t=0}\per(E_t)=0,
\end{equation}
where $\per(\cdot)$ denotes the perimeter of a set. This condition precisely signifies that an optimal set must have constant (distributional) mean curvature. Consequently, Alexandrov's Theorem characterizes balls as the sole volume-constrained critical points in the isoperimetric problem.\footnote{See also \cite{DelMag19}, where this characterization is proved to hold in the whole class of sets of finite perimeter.}

Here we consider a two-dimensional discrete version of the aforementioned result, specifically when the ambient class is restricted to all (simple) polygons with a fixed number of sides. According to the \emph{polygonal isoperimetric inequality}, the \emph{regular polygon} is the isoperimetric set in this class, its boundary having the shortest length among all polygons with the same area and same number of sides. This fact has been known since ancient times and can be proved by various methods.

A discrete version of Alexandrov's Theorem characterizes instead the regular polygon as the sole area-constrained critical point of the perimeter. In the discrete context, the families of area-preserving perturbations employed to derive the criticality condition \eqref{intro-1} must also preserve both the polygonal structure and the number of sides. One way to proceed is to identify the perimeter of an $N$-gon with a function of $2N$ real variables (the coordinates of the vertices). Then it can be demonstrated that the regular $N$-gon is the only constrained critical point of this function; an algebraic proof of this fact is presented in \cite{Bla05}, see also \cite{Bog23} for an elegant geometric argument. Notice that this result also provides a proof of the polygonal isoperimetric inequality, if one also proves the existence of an optimal polygon: see again \cite{Bog23}.

In this article we identify a \emph{minimal} class of variations that is sufficient to characterize the regular polygon. We define three particular families of perturbations of a polygon: (i) parallel movement of one side, (ii) rotation of one side around its midpoint, and (iii) movement of one vertex parallel to the line joining the two adjacent vertices. Deformations of types (i) and (ii) have already been considered in \cite{BonCriTop22,BucFra16,FraVel19}. These elementary deformations are well-suited to compute the first variation as in \eqref{intro-1} using basic calculus tools (Section~\ref{sec:crit}). Furthermore they are sufficiently general, as any variation that maintains the polygonal structure can be expressed using these deformations (Remark~\ref{rem:variations}). Additionally, they can be used to derive criticality conditions of various other shape functionals.

We show that imposing the criticality condition \eqref{intro-1} with respect to all perturbations of type (i)-(ii), or of type (ii)-(iii), characterizes the regular polygons as the sole critical polygons. We present this as our main result below (see Figure~\ref{fig:intro} for the notation that appears in the statement). We prove this result in Theorem~\ref{thm:alexandrov-pol} in Section~\ref{sec:alexandrov}.

\begin{mainresult}[Alexandrov's Theorem for polygons] \label{thm:main}
    Let $\p$ be a polygon with $N\geq3$ vertices such that for all $i\in\{1,\ldots,N\}$
    \begin{equation} \label{eq:stat-intro1}
    \frac{1}{\ell_i}\Big(\psi(\theta_i)+\psi(\theta_{i+1})\Big) = \frac{\per(\p)}{2\mathrm{Area}(\p)} \qquad\text{and}\qquad \psi(\theta_i)-\psi(\theta_{i+1})=0,
    \end{equation}
    where $\psi(\theta)\defeq \csc(\theta)+\cot(\theta)$, then $\p$ is a regular polygon.

    \smallskip
    Similarly, if $\p$ satisfies for all $i\in\{1,\ldots,N\}$
    \begin{equation} \label{eq:stat-intro2}
    \psi(\theta_i)-\psi(\theta_{i+1})=0 \qquad\text{and}\qquad \cos\alpha_i^- - \cos\alpha_i^+ = 0,
    \end{equation}
    then $\p$ is a regular polygon.
\end{mainresult}

\begin{figure}[ht]
\definecolor{qqqqff}{rgb}{0.25,0.29,0.3}
\definecolor{qqwuqq}{rgb}{1,0.55,0}
\begin{tikzpicture}[scale=0.45,line cap=round,line join=round]
\clip(-7,-2.2) rectangle (7.5,6);
\fill[line width=0pt,fill=qqwuqq!20] (-4,0) -- (4,0) -- (6,1.6) -- (5.6,4.15) -- (-5,5.6) -- (-5.88,3.6) -- cycle;
\draw [line width=1.2pt,dash pattern=on 3pt off 3pt,color=qqqqff] (-5.9,3.6)-- (-5,5.6);
\draw [line width=1.2pt,dash pattern=on 3pt off 3pt,color=qqqqff] (6,1.6)-- (5.6,4.15);
\draw [line width=1.2pt,color=qqqqff] (-5.9,3.6)-- (-4,0);
\draw [line width=1.2pt,color=qqqqff] (-4,0)-- (4,0);
\draw [line width=1.2pt,color=qqqqff] (4,0)-- (6,1.6);
\begin{footnotesize}
\draw [fill=qqqqff] (-4,0) circle (3.5pt);
\draw (-4.5,-0.2) node {$P_i$};
\draw [fill=qqqqff] (4,0) circle (3.5pt);
\draw (4.8,-0.2) node {$P_{i+1}$};
\draw [fill=qqqqff] (-5.9,3.6) circle (3.5pt);
\draw (-6.4,3.2) node {$P_{i-1}$};
\draw [fill=qqqqff] (6,1.6) circle (3.5pt);
\draw (6.8,1.3) node {$P_{i+2}$};
\draw [-,line width=0.7pt] (-3.5,0) arc (0:115:0.5);
\draw (-3.5,0.9) node {\footnotesize $\theta_i$};
\draw [-,line width=0.7pt] (3.5,0) arc (180:45:0.5);
\draw (3.5,0.9) node {\footnotesize $\theta_{i+1}$};
\draw [below] (0,0) node {\footnotesize $\ell_i$};
\end{footnotesize}
\end{tikzpicture}
\begin{tikzpicture}[scale=0.55,line cap=round,line join=round]
\clip(-6,-3) rectangle (6,5);
\fill[line width=0pt,fill=qqwuqq!20] (-3,-2) -- (-4,0) -- (-3.37,3.96) -- (4,0) -- (4.39,-2) -- cycle;
\draw [line width=1.4pt,color=qqqqff] (-4,0)-- (-3.37,3.96);
\draw [line width=1.4pt,color=qqqqff] (-3.37,3.96)-- (4,0);
\draw [line width=1.4pt,color=qqqqff] (4,0)-- (4.2,-1.026);
\draw [line width=1.4pt,color=qqqqff] (-4,0)-- (-3.5,-1);
\draw [line width=1.4pt,dash pattern=on 3pt off 3pt,color=qqqqff] (-3.5,-1)-- (-3,-2);
\draw [line width=1.4pt,dash pattern=on 3pt off 3pt,color=qqqqff] (4.2,-1.026) -- (4.39,-2);
\draw [line width=1pt,dash pattern=on 3pt off 3pt] (-4,0)-- (4,0);
\begin{footnotesize}
\draw [fill=qqqqff] (-4,0) circle (3pt);
\draw (-4.6,-0.2) node {$P_{i-1}$};
\draw [fill=qqqqff] (4,0) circle (3pt);
\draw (4.8,-0.2) node {$P_{i+1}$};
\draw [fill=qqqqff] (-3.37,3.96) circle (3pt);
\draw (-3.5,4.4) node {$P_i$};
\end{footnotesize}
\draw [-,line width=0.7pt] (-3.3,0) arc (0:80:0.7);
\draw (-3,0.7) node {\footnotesize $\alpha_{i}^-$};
\draw [-,line width=0.7pt] (3.1,0) arc (180:155:0.9);
\draw (2.2,0.4) node {\footnotesize $\alpha_{i}^+$};
\end{tikzpicture}
\caption{Notation used in the statement of the Main Result depicting the angles $\theta_i$, $\theta_{i+1}$, $\alpha_i^-$, $\alpha_i^+$, and the length $\ell_i$ of the side $\side{P_i}{P_{i+1}}$.}\label{fig:intro}
\end{figure}
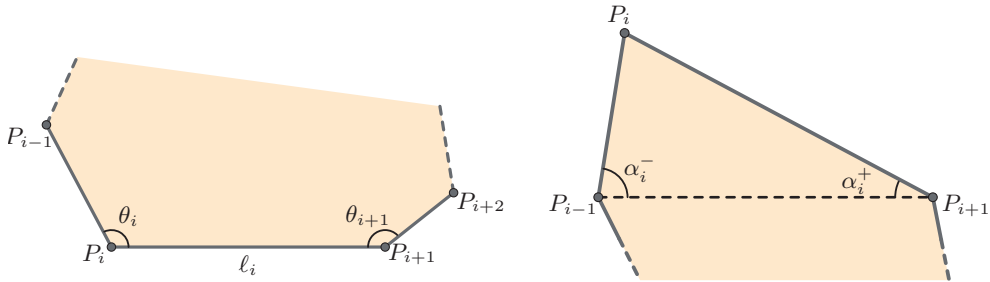

The equations in \eqref{eq:stat-intro1} correspond to the criticality conditions with respect to perturbations of type (i)-(ii), whereas the equations in \eqref{eq:stat-intro2} correspond to the criticality conditions with respect to perturbations of type (ii)-(iii), see Section~\ref{sec:crit} for their derivations. We call this theorem ``Alexandrov's Theorem for polygons'' since the conditions \eqref{eq:stat-intro1} and \eqref{eq:stat-intro2} play the role of the constant mean curvature condition of the classical Alexandov's Theorem in the discrete setting.

Our motivation for writing this article stems from various results and conjectures concerning discrete counterparts of symmetry problems having the Euclidean ball as the solution. The ball is indeed the optimal shape in numerous isoperimetric-type problems involving different functionals---examples include the fractional perimeter, the Riesz energy, the Cheeger constant, and spectral functionals such as the first Dirichlet eigenvalue of the Laplacian. For many of these problems, Alexandrov-type theorems have also been proved, not only characterizing the ball as the optimal domain, but also as the sole critical point. It becomes natural to seek discrete analogs, with a general expectation that the regular polygons should play the role of the ball in the discrete context. In the concluding section we discuss some results in the literature, as well as some open problems and conjectures. 

\section{Criticality Conditions}\label{sec:crit}

In this section we derive the criticality conditions for the perimeter functional under an area constraint, with respect to the three particular classes of perturbations of a polygon, as outlined in the Introduction.

We preliminarily fix some notation. In this paper, the term \emph{polygon} will indicate an open and bounded region of the plane $\R^2$ whose boundary is given by a closed, connected curve consisting of finitely many line segments (\emph{sides}), where only consecutive segments intersect at their endpoints (\emph{vertices}). The class of polygons with $N\geq3$ vertices is denoted by $\pol_N$. The perimeter and the area of a polygon $\p$ are denoted by $\per(\p)$ and $|\p|$, respectively.

Given two points $P,Q\in\R^2$, we denote by $\side{P}{Q}\defeq\{tP+(1-t)Q \colon t\in[0,1]\}$ the segment joining $P$ and $Q$. For $N\geq3$, let $\p\in\pol_N$ be a polygon with $N$ vertices $P_1,\ldots,P_N$. For notational convenience we also set $P_0\defeq P_N$, $P_{N+1}\defeq P_1$. For $i\in\{1,\ldots,N\}$ we let:
\begin{itemize}
\item $\nu_i$ be the exterior unit normal to the side $\side{P_i}{P_{i+1}}$,
\item $\ell_i$ be the length of the side $\side{P_i}{P_{i+1}}$,
\item $\theta_i$ be the interior angle at the vertex $P_i$,
\item $M_i$ be the midpoint of the side $\side{P_i}{P_{i+1}}$.
\end{itemize}
Given a polygon $\mathcal{P}\in\pol_N$ with $N\geq3$ vertices $P_1,\ldots,P_N$ we define the three classes of perturbations specifically as follows.

\begin{definition}[Sliding of one side] \label{def:sliding}
Fix a side $\side{P_i}{P_{i+1}}$, $i\in\{1,\ldots,N\}$. For $t\in\R$ with $|t|$ sufficiently small, we define the polygon $\mathcal{P}_t\in\pol_N$ with vertices $P_1^t,\ldots,P_N^t$ obtained as follows (see Figure~\ref{fig:sliding}):
\begin{enumerate}
\item all vertices except $P_i$ and $P_{i+1}$ are fixed, i.e.
	\[ 
		P_j^t\defeq P_j \text{ for all } j\in\{1,\ldots, N\}\setminus\{i,i+1\};
	\]
\item the vertices $P_i^t$ and $P_{i+1}^t$ lie on the lines containing $\side{P_{i-1}}{P_i}$ and $\side{P_{i+1}}{P_{i+2}}$, respectively;
\item the side $\side{P_i^t}{P_{i+1}^t}$ is parallel to $\side{P_i}{P_{i+1}}$ and at a distance $|t|$ from $\side{P_i}{P_{i+1}}$, in the direction of $\nu_i$ if $t>0$ and in the direction of $-\nu_i$ if $t<0$.
\end{enumerate}
\end{definition}

\begin{figure}[ht]
\centering
\definecolor{qqqqff}{rgb}{0.25,0.29,0.3}
\definecolor{qqwuqq}{rgb}{1,0.55,0}
\begin{tikzpicture}[scale=0.45,line cap=round,line join=round]
\clip(-7,-2.2) rectangle (7.5,6);
\fill[line width=0pt,fill=qqwuqq!20] (-3.5,-0.96) -- (2.7,-0.96) -- (6,1.6) -- (5.6,4.15) -- (-5,5.6) -- (-5.88,3.6) -- cycle;
\draw [line width=1.2pt,dash pattern=on 3pt off 3pt,color=qqqqff] (-5.9,3.6)-- (-5,5.6);
\draw [line width=1.2pt,dash pattern=on 3pt off 3pt,color=qqqqff] (6,1.6)-- (5.6,4.15);
\draw [line width=1.2pt,color=qqqqff] (-5.9,3.6)-- (-4,0);
\draw [line width=1.2pt,color=qqqqff] (-4,0)-- (4,0);
\draw [line width=1.2pt,color=qqqqff] (4,0)-- (6,1.6);
\draw [line width=1.2pt,color=qqwuqq] (-4,0)-- (-3.5,-0.96);
\draw [line width=1.2pt,color=qqwuqq] (4,0)-- (2.7,-0.96);
\draw [line width=1.2pt,color=qqwuqq] (2.7,-0.96)-- (-3.5,-0.96);
\draw [-stealth,line width=1.2pt] (0,0) -- (0,-2);
\draw [->,line width=0.7pt] (-2.253649257845337,0)-- (-2.253649257845337,-0.96);
\begin{footnotesize}
\draw [fill=qqqqff] (-4,0) circle (3.5pt);
\draw (-4.5,-0.2) node {$P_i$};
\draw [fill=qqqqff] (4,0) circle (3.5pt);
\draw (4.8,-0.2) node {$P_{i+1}$};
\draw [fill=qqqqff] (-5.9,3.6) circle (3.5pt);
\draw (-6.4,3.2) node {$P_{i-1}$};
\draw [fill=qqqqff] (6,1.6) circle (3.5pt);
\draw (6.8,1.3) node {$P_{i+2}$};
\draw [fill=qqwuqq] (-3.5,-0.96) circle (3.5pt);
\draw (-4.1,-1.2) node {$P_i^t$};
\draw [fill=qqwuqq] (2.7,-0.96) circle (3.5pt);
\draw (3.5,-1.2) node {$P_{i+1}^t$};
\draw[color=black] (0.7,-2) node {$\nu_i$};
\draw[color=black] (-1.75,-0.5) node {\tiny $t$};
\draw (0,3) node {$\p_t$};
\draw [-,line width=0.7pt] (-3.5,0) arc (0:115:0.5);
\draw (-3.5,0.7) node {\tiny $\theta_i$};
\draw [-,line width=0.7pt] (3.5,0) arc (180:45:0.5);
\draw (3.5,0.8) node {\tiny $\theta_{i+1}$};
\end{footnotesize}
\end{tikzpicture}
\caption{A polygon $\p$ and its variation $\p_t$ (shaded region) as in Definition~\ref{def:sliding}, obtained by sliding the side $\side{P_i}{P_{i+1}}$ in the normal direction at a distance $t>0$.}\label{fig:sliding}
\end{figure}
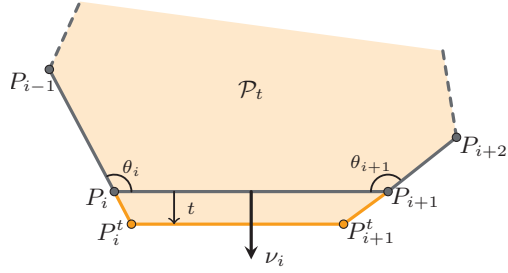

\begin{definition}[Tilting of one side] \label{def:tilting}
Fix a side $\side{P_i}{P_{i+1}}$, $i\in\{1,\ldots,N\}$. For $t\in\R$ with $|t|$ sufficiently small, we define the polygon $\mathcal{P}_t\in\pol_N$ with vertices $P_1^t,\ldots,P_N^t$ obtained as follows (see Figure~\ref{fig:tilting}):
\begin{enumerate}
\item all vertices except $P_i$ and $P_{i+1}$ are fixed, i.e.
	\[
		P_j^t\defeq P_j \text{ for all } j\in\{1,\ldots, N\}\setminus\{i,i+1\};
	\]
\item the vertices $P_i^t$ and $P_{i+1}^t$ lie on the lines containing $\side{P_{i-1}}{P_i}$ and $\side{P_{i+1}}{P_{i+2}}$, respectively;
\item the line containing $\side{P_i^t}{P_{i+1}^t}$ is obtained by rotating the line containing $\side{P_i}{P_{i+1}}$ around the midpoint $M_i$ of $\side{P_i}{P_{i+1}}$ by an angle of amplitude $|t|$;
\item the direction of rotation is such that for $t>0$ the angle $\theta_i$ is decreased by $|t|$ and $\theta_{i+1}$ is increased by $|t|$, whereas for $t<0$ the angle $\theta_i$ is increased by $|t|$ and $\theta_{i+1}$ is decreased by $|t|$.
\end{enumerate}
\end{definition}

\begin{figure}[ht]
\centering
\definecolor{qqqqff}{rgb}{0.25,0.29,0.3}
\definecolor{qqwuqq}{rgb}{1,0.55,0}
\begin{tikzpicture}[scale=0.55,line cap=round,line join=round]
\clip(-7,-1.8) rectangle (7.5,6);
\fill[line width=0pt,fill=qqwuqq!20] (-3.45,-1.04) -- (5.28,1.58) -- (6,2.48) -- (5.57,4.17) -- (-5,5.67) -- (-5.94,3.66) -- cycle;
\draw [line width=1.2pt,dash pattern=on 3pt off 3pt,color=qqqqff] (-5.94,3.66)-- (-5,5.67);
\draw [line width=1.2pt,dash pattern=on 3pt off 3pt,color=qqqqff] (6,2.48)-- (5.57,4.17);
\draw [line width=1.2pt,color=qqqqff] (-5.94,3.66)-- (-4,0);
\draw [line width=1.2pt,color=qqqqff] (-4,0)-- (4,0);
\draw [line width=1.2pt,color=qqqqff] (4,0)-- (6,2.48);
\draw [line width=1.2pt,color=qqwuqq] (-4,0)-- (-3.449980437567673,-1.0361968117094946);
\draw [line width=1.2pt,color=qqwuqq] (-3.449980437567673,-1.0361968117094946)-- (5.2792724920825576,1.585622128425266);
\draw [line width=1.2pt,color=qqwuqq] (5.2792724920825576,1.585622128425266)-- (6,2.48);
\begin{footnotesize}
\draw [fill=qqqqff] (-4,0) circle (3.5pt);
\draw (-4.4,-0.2) node {$P_i$};
\draw [fill=qqqqff] (4,0) circle (3.5pt);
\draw (4.7,-0.2) node {$P_{i+1}$};
\draw [fill=qqqqff] (-5.94,3.66) circle (3.5pt);
\draw (-6.3,3.3) node {$P_{i-1}$};
\draw [fill=qqqqff] (6,2.48) circle (3.5pt);
\draw (6.7,2.4) node {$P_{i+2}$};
\draw [fill] (0,0) circle (3.5pt);
\draw (0.3,-0.4) node {$M_i$};
\draw [fill=qqwuqq] (5.28,1.58) circle (3.5pt);
\draw (5.9,1.4) node {$P_{i+1}^t$};
\draw [fill=qqwuqq] (-3.45,-1.04) circle (3.5pt);
\draw (-3.9,-1.4) node {$P_i^t$};
\draw (0,3) node {$\p_t$};
\draw [->,line width=0.7pt] (1.8,0) arc (0:17:1.8);
\draw (2.1,0.3) node {\tiny $t$};
\draw [-,line width=0.7pt] (-3.5,0) arc (0:115:0.5);
\draw (-3.5,0.7) node {\tiny $\theta_i$};
\draw [-,line width=0.7pt] (3.5,0) arc (180:45:0.5);
\draw (3.5,0.7) node {\tiny $\theta_{i+1}$};
\end{footnotesize}
\end{tikzpicture}
\caption{A polygon $\p$ and its variation $\p_t$ (shaded region) as in Definition~\ref{def:tilting}, obtained by tilting the side $\side{P_i}{P_{i+1}}$ around its midpoint $M_i$ by an angle $t>0$.}\label{fig:tilting}
\end{figure}
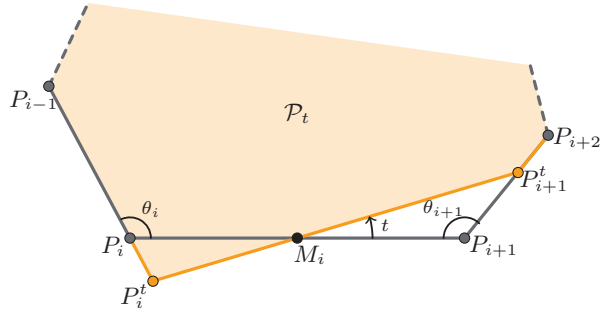

\begin{definition}[Moving of one vertex] \label{def:ricIvar}
Fix three consecutive vertices $P_{i-1}$, $P_i$, $P_{i+1}$, $i\in\{1,\ldots,N\}$, of the polygon $\p$. For $t\in\R$ with $|t|$ sufficiently small, we define the polygon $\mathcal{P}_t\in\pol_N$ with vertices $P_1^t,\ldots,P_N^t$ obtained as follows (see Figure~\ref{fig:ricIvar}):
\begin{enumerate}
\item all vertices except $P_i$ are fixed, i.e., $P_j^t\defeq P_j$ for all $j\in\{1,\ldots, N\}\setminus\{i\}$;
\item the vertex $P_i^t$ is given by
$$
P_i^t = P_i + t\,\frac{P_{i+1}-P_{i-1}}{|P_{i+1}-P_{i-1}|},
$$
that is, $P_i^t$ lies on the line through $P_i$ parallel to the diagonal $\side{P_{i-1}}{P_{i+1}}$, at a distance $|t|$ from $P_i$.
\end{enumerate}
\end{definition}

\begin{figure}[ht]
\centering
\definecolor{qqqqff}{rgb}{0.25,0.29,0.3}
\definecolor{qqwuqq}{rgb}{1,0.55,0}
\begin{tikzpicture}[scale=0.65,line cap=round,line join=round]
\clip(-6,-3) rectangle (6,5);
\fill[line width=0pt,fill=qqwuqq!20] (-2.622,-2.756) -- (-4,0) -- (-2,3.96) -- (4,0) -- (4.54,-2.76) -- cycle;
\draw [line width=1.4pt,color=qqqqff] (-4,0)-- (-3.37,3.96);
\draw [line width=1.4pt,color=qqqqff] (-3.37,3.96)-- (4,0);
\draw [line width=1.4pt,color=qqqqff] (4,0)-- (4.39,-2);
\draw [line width=1.4pt,color=qqqqff] (-4,0)-- (-3,-2);
\draw [line width=1.4pt,dash pattern=on 3pt off 3pt,color=qqqqff] (-3,-2)-- (-2.622,-2.756);
\draw [line width=1.4pt,dash pattern=on 3pt off 3pt,color=qqqqff] (4.39,-2)-- (4.54,-2.76);
\draw [line width=1.4pt,color=qqwuqq] (-4,0)-- (-2,3.96);
\draw [line width=1.4pt,color=qqwuqq] (-2,3.96)-- (4,0);
\draw [line width=1pt,dash pattern=on 3pt off 3pt] (-4,0)-- (4,0);
\begin{footnotesize}
\draw [fill=qqqqff] (-4,0) circle (3pt);
\draw (-4.5,-0.2) node {$P_{i-1}$};
\draw [fill=qqqqff] (4,0) circle (3pt);
\draw (4.6,-0.2) node {$P_{i+1}$};
\draw [fill=qqqqff] (-3.37,3.96) circle (3pt);
\draw (-3.5,4.25) node {$P_i$};
\draw [fill=qqwuqq] (-2,3.96) circle (3pt);
\draw (-1.7,4.25) node {$P_i^t$};
\draw (-2.6,4.2) node {$t$};
\end{footnotesize}
\draw [->,line width=1pt] (-3.37,3.96)-- (-2,3.96);
\draw [-,line width=0.7pt] (-3.3,0) arc (0:80:0.7);
\draw (-3,0.5) node {\tiny $\alpha_{i}^-$};
\draw [-,line width=0.7pt] (3.1,0) arc (180:155:0.9);
\draw (2.7,0.25) node {\tiny $\alpha_{i}^+$};
\end{tikzpicture}
\caption{A polygon $\p$ and its variation $\p_t$ (shaded region) as in Definition~\ref{def:ricIvar}, obtained by moving the vertex $P_i$ parallel to the diagonal $\side{P_{i-1}}{P_{i+1}}$ at a distance $t>0$.}\label{fig:ricIvar}
\end{figure}
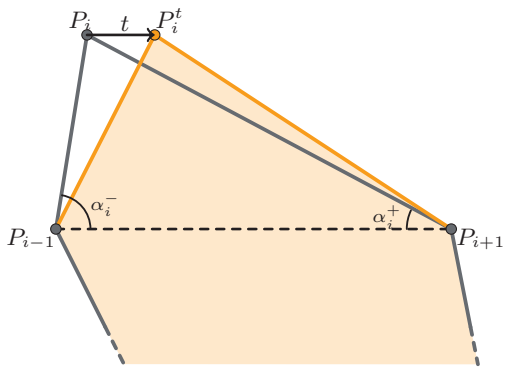

\begin{definition}[Stationarity] \label{def:stationarity}
Let $\p\in\pol_N$ and let $\{\p_t\}_t$ be a one-parameter deformation of $\p$, such as those considered before. We say that $\p$ is \emph{stationary (for the perimeter functional) with respect to the variation $\{\p_t\}_t$ under area constraint} if
\begin{equation} \label{eq:stationary}
\frac{\dd}{\dd t}\left(\frac{\per(\p_t)}{|\p_t|^{1/2}} \right) \bigg|_{t=0} = 0.
\end{equation}
\end{definition}

In the following theorem we derive the stationarity conditions under area constraint for a polygon $\p\in\pol_N$ with respect to the previous three classes of perturbations. Two of these conditions are expressed in terms of the function
\begin{equation} \label{eq:psi}
\psi(\theta)\defeq \frac{1}{\sin(\theta)}+\cot(\theta).
\end{equation}

\begin{theorem}[Stationarity conditions] \label{thm:stationarity}
A polygon $\p\in\pol_N$ is stationary with respect to the sliding variation as in Definition~\ref{def:sliding} on the $i$-th side, for $i\in\{1,\ldots,N\}$, under area constraint if and only if
\begin{equation}\label{eq:stat_sliding}
\frac{1}{\ell_i}\Big(\psi(\theta_i)+\psi(\theta_{i+1})\Big) = \frac{\per(\p)}{2|\p|}\, ,
\end{equation}
where $\psi$ is defined in \eqref{eq:psi}.

A polygon $\p\in\pol_N$ is stationary with respect to the tilting variation as in Definition~\ref{def:tilting} on the $i$-th side, for $i\in\{1,\ldots,N\}$, under area constraint if and only if
\begin{equation}\label{eq:stat_tilting}
\psi(\theta_i)-\psi(\theta_{i+1})=0.
\end{equation}

A polygon $\p\in\pol_N$ is stationary with respect to the moving of the $i$-th vertex as in Definition~\ref{def:ricIvar}, for $i\in\{1,\ldots,N\}$, under area constraint if and only if
\begin{equation}\label{eq:stat_ricIvar}
\cos\alpha_i^- - \cos\alpha_i^+ = 0,
\end{equation}
where $\alpha_{i}^-\in(0,\pi)$ is the angle between $\side{P_{i-1}}{P_{i+1}}$ and $\side{P_{i-1}}{P_i}$, and $\alpha_{i}^+\in(0,\pi)$ is the angle between $\side{P_{i-1}}{P_{i+1}}$ and $\side{P_i}{P_{i+1}}$, see Figure~\ref{fig:ricIvar}.
\end{theorem}

\begin{proof}
To obtain the stationarity conditions, we first express the area and the perimeter of the perturbed polygon $\p_t$ as a function of the variable $t$ (up the first order); we then differentiate the quotient in \eqref{eq:stationary} with respect to $t$ and set it equal to zero.

As in \cite[pp. 103–106]{BucFra16}, the first variations of the area and of the perimeter with respect to the sliding perturbation in Definition~\ref{def:sliding} can be obtained from the identities
$$
|\p_t| = |\p| + \ell_i t + o(t), \qquad \per(\p_t)= \per(\p) +t\bigl(\psi(\theta_i)+\psi(\theta_{i+1})\bigr)
$$
where $\frac{o(t)}{t}\to0$ as $t\to0$. These formulas are simple consequences of geometric and trigonometric arguments. Hence,
    \[
        \frac{\dd}{\dd t}\left(\frac{\per(\p_t)}{|\p_t|^{1/2}} \right) \bigg|_{t=0} = \frac{1}{|\p|}\left(\big(\psi(\theta_i)+\psi(\theta_{i+1})\big)|\p|^{1/2} - \frac{\ell_i}{2|\p|^{1/2}}\per(\p)\right)=0
    \]
implies condition \eqref{eq:stat_sliding}.

Similarly, the first variations of the area and of the perimeter with respect to the tilting perturbation in Definition~\ref{def:tilting} are obtained from the identities (see also \cite[pp.~103–106]{BucFra16})
\begin{gather*}
|\p_t| = |\p| + o(t), \\ \per(\p_t) = \per(\p) - \ell_i + \frac{\ell_i}{2}\biggl(\frac{\sin\theta_{i+1}-\sin t}{\sin(\theta_{i+1}+t)} + \frac{\sin\theta_i+\sin t}{\sin(\theta_i-t)}\biggr).
\end{gather*}
Namely, 
    \begin{align*}
        \frac{\dd}{\dd t}\left(\frac{\per(\p_t)}{|\p_t|^{1/2}} \right) \bigg|_{t=0} &= \frac{\ell_i}{2|\p|^{1/2}}\frac{\dd}{\dd t} \biggl(\frac{\sin\theta_{i+1}-\sin t}{\sin(\theta_{i+1}+t)} + \frac{\sin\theta_i+\sin t}{\sin(\theta_i-t)}\biggr) \bigg|_{t=0} \\
                     &= \frac{\ell_i}{2|\p|^{1/2}}\biggl( \frac{1}{\sin\theta_i}
                        +\cot\theta_i - \frac{1}{\sin\theta_{i+1}} 
                            - \cot\theta_{i+1} \biggr) = 0,
    \end{align*}
and the condition \eqref{eq:stat_tilting} follows.

Finally, the first variations of the area and of the perimeter with respect to the perturbation in Definition~\ref{def:ricIvar} follow from elementary geometric arguments. Since this perturbation is area preserving we have that $|\p_t|=|\p|$. On the other hand,
\begin{align*}
\per(\p_t)
& = \per(\p) + \sqrt{\ell_{i-1}^2+2t\ell_{i-1}\cos\alpha_i^-+t^2} - \ell_{i-1} \\
&\qquad\qquad\qquad\qquad\qquad\qquad + \sqrt{\ell_i^2 -2t\ell_i\cos\alpha_i^++t^2}-\ell_i \\
& = \per(\p) + t (\cos\alpha_{i}^--\cos\alpha_{i}^+) + o(t)
\end{align*}
as $t\to0$. Hence, from
    \begin{equation*}
        \frac{\dd}{\dd t}\left(\frac{\per(\p_t)}{|\p_t|^{1/2}} \right) \bigg|_{t=0} = \frac{1}{|\p|^{1/2}}\bigl(\cos\alpha_{i}^--\cos\alpha_{i}^+\bigr) = 0
    \end{equation*}
the condition \eqref{eq:stat_ricIvar} follows.
\end{proof}

\begin{remark}\label{rem:variations} We observe that \textit{any} variation of a polygon $\p\in\pol_N$ can be expressed in terms of the sliding and tilting variations as in Definitions~\ref{def:sliding}--\ref{def:tilting}. Indeed, let $\p'\in\pol_N$ be any polygon with vertices $\{P_1',\ldots,P_N'\}$ sufficiently close to those of $\p$. To prove the property, by iteration it is enough to consider the case where $\p$ and $\p'$ differ only by one vertex, say $P_i$ (hence $P_j=P_j'$ for all $j\neq i$).

We first observe that given a side $\overline{P_jP_{j+1}}$ we can define a family of variations, similar to the tilting perturbation in Definition~\ref{def:tilting}, by rotating one side with respect to one of its endpoints, say $P_j$ (so that $P_j$ remains fixed and $P_{j+1}$ moves along the line containing the segment $\overline{P_{j+1}P_{j+2}}$). Such a variation can be easily obtained as the result of a composition of our sliding and tilting variations.

Now, if $\p$ and $\p'$ differ only by the $i$-th vertex, we can first rotate the side $\overline{P_{i-1}P_i}$ around $P_{i-1}$, so that the rotated side is contained in the line passing through $P_{i-1}$ and $P_i'$. Then we rotate the side $\overline{P_iP_{i+1}}$ around the point $P_{i+1}$ to align it with $\overline{P_i'P_{i+1}}$. After these two variations the polygon $\p$ is transformed into $\p'$. Since by the previous observation the rotation about a vertex is a combination of sliding and tilting variations, the claim is proved.
\end{remark}
 
\section{Alexandrov's Theorem for Polygons}\label{sec:alexandrov}

In the following theorem we observe that the stationarity conditions with respect to two of the three families of perturbations considered in Theorem~\ref{thm:stationarity} (namely, sliding \& tilting, or tilting \& moving of one vertex) uniquely characterize the regular polygons.

\begin{theorem}[Alexandrov's Theorem for polygons]
\label{thm:alexandrov-pol}
    If $\p\in\pol_N$ satisfies conditions \eqref{eq:stat_sliding}--\eqref{eq:stat_tilting}, or conditions \eqref{eq:stat_tilting}--\eqref{eq:stat_ricIvar}, for all $i\in\{1,\ldots,N\}$, then $\p$ is a regular polygon.
\end{theorem}

\begin{proof}
    Suppose \eqref{eq:stat_sliding} and \eqref{eq:stat_tilting} hold for all $i\in\{1,\ldots,N\}$. Then, by \eqref{eq:stat_tilting}, we have that $\psi(\theta_i)=\lambda$ for some $\lambda\in\R$ and for all $i\in\{1,\ldots,N\}$. Hence the condition \eqref{eq:stat_sliding} yields
        \[
         \ell_i = \frac{4\lambda|\p|}{\per(\p)} \qquad \text{for all } i\in\{1,\ldots,N\},
        \]
    i.e., $\p$ is equilateral. Now, since $\psi^\pr(\theta) < 0$ on $(0,2\pi)$ the function $\psi$ is injective, and \eqref{eq:stat_tilting} yields $\theta_i=\theta_{i+1}$ for all $i\in\{1,\ldots,N\}$, i.e., $\p$ is equiangular.

    Suppose \eqref{eq:stat_tilting} and \eqref{eq:stat_ricIvar} hold for all $i\in\{1,\ldots,N\}$. Since the cosine function is injective on $(0,\pi)$, the condition \eqref{eq:stat_ricIvar} implies that $\alpha_i^- = \alpha_i^+$ for all $i\in\{1,\ldots,N\}$, hence, $\p$ is equilateral. Again, by \eqref{eq:stat_tilting} we obtain that $\p$ is equiangular.
\end{proof}

\begin{remark}\label{rmk:triangles}
In the case $N=3$ the stationarity condition \eqref{eq:stat_sliding} with respect to the sliding variation is always satisfied by any triangle. Indeed, if $\{\p_t\}$ is the variation in Definition~\ref{def:sliding} of a triangle $\p$, then $\p_t$ is a triangle similar to $\p$, so that scaling $\p_t$ by a factor $(\frac{|\p|}{|\p_t|})^{1/2}$ gives back the starting triangle $\p$; hence
\begin{equation*}
    0 = \frac{\dd}{\dd t}\bigg|_{t=0} \per\biggl( \biggl(\frac{|\p|}{|\p_t|}\biggr)^{\frac12} \p_t \biggr)
    = |\p|^{\frac12} \frac{\dd}{\dd t}\left(\frac{\per(\p_t)}{|\p_t|^{1/2}} \right)\bigg|_{t=0}
\end{equation*}
so that the stationarity condition \eqref{eq:stationary} is always satisfied for this variation.

The equilateral triangle is then characterized either by the sole condition \eqref{eq:stat_tilting} or by the sole condition \eqref{eq:stat_ricIvar}. Notice also that imposing \eqref{eq:stat_tilting} on a single side (or \eqref{eq:stat_ricIvar} on a single vertex) yields an isosceles triangle; therefore to characterize the equilateral triangle it is sufficient to impose condition \eqref{eq:stat_tilting} only on two sides (or \eqref{eq:stat_ricIvar} only on two vertices).
\end{remark}

\begin{remark}\label{rmk:equilat-equiang}
From the proof of Theorem~\ref{thm:alexandrov-pol} it is clear that the stationarity conditions with respect to the tilting variation and with respect to the movement of one vertex characterize the class of equiangular and equilateral polygons, respectively. More precisely, we have that for a polygon $\p\in\pol_N$ with $N\geq3$:
\begin{itemize}
    \item $\p$ satisfies \eqref{eq:stat_tilting} for all $i\in\{1,\ldots,N\}$ if and only if $\p$ is equiangular;
    \item $\p$ satisfies \eqref{eq:stat_ricIvar} for all $i\in\{1,\ldots,N\}$ if and only if $\p$ is equilateral.
\end{itemize}
It is an open question as to whether it is possible to characterize by a similar geometric condition the class of polygons $\p\in\pol_N$ which obey the criticality condition \eqref{eq:stat_sliding} with respect to the sliding variation. For $N=3$, all triangles satisfy \eqref{eq:stat_sliding}, as observed in Remark~\ref{rmk:triangles}. For $N=4$, \eqref{eq:stat_sliding} is satisfied by all kites (i.e., quadrilaterals symmetric with respect to their reflection across at least one diagonal). It is an open question as to whether there are other quadrilaterals satisfying \eqref{eq:stat_sliding}. For general $N$ even, all $N$-gons which are reflection-symmetric with respect to the bisectors of their angles satisfy \eqref{eq:stat_sliding}.
\end{remark}

\section{Further Results and Conjectures}\label{sec:questions}

Several classical functionals from shape optimization share with the Euclidean perimeter the property that the only optimal domains are balls. Among these functionals, of paramount importance are the torsional rigidity, the principal (first Dirichlet) eigenvalue of the Laplacian, or the logarithmic capacity, defined for $\Omega\subset\R^d$ as 
\begin{equation*}
    \tau(\Omega)\defeq - \inf_{u\in H^1_0(\Omega)}\int_\Omega \bigl(|\nabla u|^2-2u\bigr)\dd x,
    \qquad
    \lambda_1(\Omega)\defeq \inf_{u\in H^1_0(\Omega)\setminus\{0\}}\frac{\int_\Omega|\nabla u|^2\dd x}{\int_\Omega u^2\dd x},
\end{equation*}
\begin{equation*}
     \capp(\Omega) \defeq \exp \bigl( - \lim_{|x|\to\infty} (u(x) - \log |x|)\bigr),
\end{equation*}
respectively. Here $H^1_0(\Omega)$ denotes the space of functions in the Sobolev space $H^1(\Omega)=W^{1,2}(\Omega)$ which vanish on the boundary of $\Omega$. In the definition of $\capp$ the function $u$ is the log-equilibrium potential of $\Omega$ and satisfies $\Delta u = 0$ in $\R^2 \setminus \Omega$, $u = 0$ on $\pt\Omega$, and $u(x) \sim \log |x|$ as $|x| \to  +\infty$. It is then natural to look at discrete problems where these functionals are restricted to the class of polygons with a fixed number of sides.
In their seminal monograph \cite{PolSze51}, George P\'{o}lya and G\'{a}bor Szeg\H{o} conjectured that regular polygons are optimal for the torsional rigidity $\tau$ and the principal eigenvalue of the Laplacian $\lambda_1$. They proved this conjecture for triangles and quadrilaterals using Steiner symmetrization. Whether regular $N$-gons with $N \geq 5$ are optimal for these two functionals are considered as important open problems in shape optimization. One can further wonder whether the regular polygon is characterized by the stationarity conditions with respect to the families of perturbations as defined in Section~\ref{sec:crit}, i.e., whether a discrete Alexandrov-type theorem for these spectral functionals holds. To date, the optimality of the regular polygon \emph{for every} $N\geq 3$ has only been obtained by Alexander Yu. Solynin and Victor Zalgaller \cite{SolZal04} for the logarithmic capacity $\capp$, and by Dorin Bucur and Ilaria Fragal\`a \cite{BucFra16} for the Cheeger constant
\begin{equation*}
    h(\Omega) \defeq \inf\Bigl\{\frac{\Per(A;\R^2)}{|A|} \colon A\subset\Omega \text{ measurable}\Bigr\}.
\end{equation*}
 Furthermore, Ilaria Fragal\`a and Bozhidar Velichkov \cite{FraVel19} showed that equilateral triangles are characterized as the sole critical points of $\tau$ and $\lambda_1$ with respect to the tilting variations as defined in Definition~\ref{def:tilting}.

Recently, optimization over polygons of nonlocal interaction functionals such as the fractional perimeter or Riesz-type energies
\begin{equation*}
    \Per_s(\Omega)\defeq \int_{\Omega}\int_{\R^d\setminus\Omega}\frac{\dd x \dd y}{|x-y|^{d+s}},
    \qquad
    \mathfrak{R}(\Omega)\defeq \int_{\Omega}\int_{\Omega} K(|x-y|)\dd x \dd y
\end{equation*}
have also attracted interest. Here $s\in(0,1)$ and $K$ is a nonnegative function such that $r\mapsto r^{d-1}K(r)$ is locally integrable on $\R$. When $K$ is strictly decreasing and $C^1$, in \cite{BonCriTop22} we observed that in this case P\'{o}lya and Szeg\H{o}'s argument allows one to conclude that among triangles and quadrilaterals the regular polygon maximizes $\mathfrak{R}$. In the same paper we conjectured that this result holds for every regular polygon with $N\geq 3$ when the energy functional is defined via Riesz kernels $K(|x|)=|x|^{-\alpha}$ with $0<\alpha<2$. Quite surprisingly, Beniamin Bogosel, Dorin Bucur, and Ilaria Fragal\`a \cite{BogBucFra23} recently showed that this conjecture is false for more general kernels. Indeed, for Riesz-type kernels with positive powers, i.e., for $K(|x|)=-|x|^k$ with $k>0$, they showed that for even $N\geq 6$, there exists a critical $\bar{k}$ such that for $k\geq\bar{k}$ the regular polygon is not the maximizer of the Riesz-type energy $\mathfrak{R}$. An analogous property is proved for characteristic kernels $K(|x|)=\chi_{[0,r]}(|x|)$ for suitable $r$ (depending on $N$). Interestingly, only for $k=2$ and $k=4$ were they able to prove that the regular $N$-gon minimizes $\mathfrak{R}$ among \emph{all} $N$-gons with $N\geq 3$ via a polygonal Hardy-Littlewood inequality.

Related to Alexandrov's Soap Bubble Theorem, in \cite{BonCriTop22} we also showed that, under an area or a perimeter constraint, the equilateral triangle and the square are the only stationary polygons with $N=3$ and $N=4$ sides, respectively, with respect to the sliding and tilting deformations in Definitions~\ref{def:sliding} and \ref{def:tilting}; a proof in the general case $N\geq 5$ is still missing. We also mention that the same rigidity theorem has been proved in \cite{BogBucFra23} for all $N\geq3$ for characteristic kernels $K(|x|)=\chi_{[0,r]}(|x|)$ with sufficiently small support (depending on $N$, which in some sense makes the problem more local).

Obtaining the minimality of the regular polygon for the functional $\Per_s$ as well as its characterization as the only critical point with respect to certain classes of perturbations are further open problems.



\subsection*{Acknowledgments} MB is member of the GNAMPA group of INdAM. IT is partially supported by a Simons Collaboration grant 851065 and an NSF grant DMS 2306962.

\bibliographystyle{amsalpha}
\def\MR#1{}
\bibliography{references}

\end{document}